\documentclass[11pt]{article}
\usepackage{amssymb}
\usepackage{amsthm}
\usepackage{amsmath}
\usepackage{fullpage,hyperref}
\usepackage{color,graphics}

\usepackage{epsfig}
\usepackage{graphicx} %%Used for pdf version

\newtheorem{theorem}{Theorem}[section]
\newtheorem{lemma}[theorem]{Lemma}

\newtheorem{corollary}[theorem]{Corollary}

\newtheorem{definition}[theorem]{Definition}

\newtheorem{fact}[theorem]{Fact}

\newtheorem{remark}{Remark}[section]
\newtheorem{notation}{Notation}[section]

\newtheorem{conjecture}{Conjecture}
\newtheorem{question}{Question}

\newcommand{\R}{\ensuremath{\mathbb{R}}}

\newcommand{\1}{\mathbf{1}}

\newcommand{\thmref}[1]{Thm.~\ref{#1}}
\newcommand{\lemref}[1]{Lemma~\ref{#1}}

\newcommand{\sym}{\mathrm{sym}}
\newcommand{\remove}[1]{}

\title{{\bf Eigenvectors of random graphs: Nodal domains}}

\author{Yael Dekel
\thanks{Work supported in part by a grant from the binational science foundation US-Israel.}
\\ {\small Hebrew University}
\and James R. Lee
\thanks{Much of this work was done during a visit of the author to the Hebrew University.  Research partially supported by NSF CCF-0644037.}
\\ {\small University of Washington}
\and Nathan Linial
\thanks{Work supported in part by a grant from the binational science foundation US-Israel.}
\\ {\small Hebrew University} }

\begin{document}

\date{}
\maketitle

\begin{abstract}
 We initiate a systematic study of eigenvectors of random graphs.
 Whereas much is known about eigenvalues of graphs and how they
 reflect properties of the underlying graph, relatively little is
 known about the corresponding eigenvectors. Our main focus in this
 paper is on the {\em nodal domains} associated with the different
 eigenfunctions. In the analogous realm of Laplacians of Riemannian
 manifolds, nodal domains have been the subject of intensive research
 for well over a hundred years. Graphical nodal domains turn out to
 have interesting and unexpected properties. Our main theorem asserts
 that there is a constant $c$ such that for almost every graph $G$,
 each eigenfunction of $G$ has at most two large nodal domains, and
 in addition at most $c$ exceptional vertices outside these primary
 domains. We also discuss variations of these questions and briefly
 report on some numerical experiments which, in particular, suggest
 that almost surely there are just two nodal domains and no
 exceptional vertices.
\end{abstract}

\section{Introduction}

Let $G$ be a graph and let $A$ be its adjacency matrix. The
eigenvalues of $A$ turn out to encode a good deal of interesting
information about the graph $G$. Such phenomena have been intensively
investigated for over half a century. We refer the reader to the
book~\cite[Ch. 11]{lovasz} for a general discussion of this subject
and to the survey article~\cite{hlw} for the connection between
eigenvalues and expansion.  Strangely, perhaps, not much is known
about the {\em eigenvectors} of $A$ and how they are related to the
properties of $G$. However, in many application areas such as machine
learning and computer vision, eigenvectors of graphs are being used
with great success in various computational tasks such as partitioning
and clustering. For example, see the work of Shi and Malik
\cite{SM00}, Coifman, et. al. \cite{Coifman1,Coifman2}, Pothen, Simon
and Liou \cite{PSL90} and others.  In particular, a basic technique
for spectral partitioning (e.g. Weiss~\cite{Weiss99}) involves
splitting a graph according to its nodal domains. As far as we know,
the success of these methods has not yet been given a satisfactory
theoretical explanation and we hope that our investigations will help
in shedding some light on these issues as well.
We also mention that nodal domain counts in graphs
are relevant to various studies in statistical physics; see the survey \cite{howtocount}.

There is, on the other hand, a rich mathematical theory dealing with
the spectrum and eigenfunctions of Laplacians on manifolds. We only
mention this important and highly relevant background material and
refer the reader who wants to know more about this theory to Chapter 8
in Marcel Berger's monumental panorama of Riemannian
Geometry~\cite{berger}.  Suffices it to say that the adjacency matrix
of a graph is a discrete analogue of the Laplacian (we say a bit more
about other analogs below).
The geometric perspective of graphs is, in our opinion, among the most promising and most exciting aspects of present-day graph theory. Among the recent successes of this point of view is the metric theory of graphs and its computational applications. The success of the metric theory notwithstanding, we believe that other areas of geometry can be incorporated into this line of development.

The only necessary facts we require are that the
(geometric) Laplacian has a discrete spectrum, that its smallest
eigenvalue is zero and that the corresponding eigenfunction is the
constant function. This is analogous to the fact that the first
eigenvector of a finite connected graph is a positive vector and in
particular, if the graph at hand is $d$-regular, then its first
eigenvalue is $d$ and that in the corresponding eigenvector all
coordinates are equal.

{\em Nodal domains} of eigenfunctions of the Laplacian have been
studied in depth for more than a century.  We will only discuss this
concept in the realm of graphs and refer the interested reader
to~\cite{berger} for further information about the geometric
setting.   So what are nodal domains? Let $G$ be a finite connected
graph. It is well-known that every eigenfunction $f$ but the first
takes both positive and negative values. These values induce a
partition of the vertex set $V(G)$ into maximal connected components
on which $f$ does not change its sign; these are the nodal domains of $f$
(see below for the precise definition).

We maintain the following convention: If $G$ is an $n$-vertex
graph, we denote the eigenvalues of its adjacency matrix by $\lambda_1
\ge \lambda_2 \ge \ldots$. We simply refer to the $\lambda_i$ as the
{\em eigenvalues of $G$} and let $f_1, f_2, \ldots$ be the
corresponding eigenfunctions, normalized in $\ell_2$.
A slight adaptation of classical theorem
due to Courant (see, for example,~\cite{chavel}) shows that for every
$k$, the eigenfunction $f_k$ has at most $k$ nodal domains. This
statement is a bit inaccurate and we refer the reader
to~\cite{discretenodal} for a full account of Courant's theorem for
graphs.

We impose throughout some fixed (but arbitrary) ordering on the vertex
set $V=\{v_1, v_2, \ldots, v_n\}$, andthe coordinates in eigenvectors
of $G$ are arranged in this order. We freely interchange between the
vector $(f(v_1), \ldots, f(v_n))$ and the corresponding function $f :
V \to \mathbb R$.  In the graph setting, one has to be careful in
defining nodal domains properly.

\begin{definition}
 Let $G=(V,E)$ be a graph and let $f: V \rightarrow \R$ be any real
 function.  A subset $D \subseteq V$ is a {\em weak nodal domain of
   $f$} if it is a maximal subset of $V$ subject to the two
 conditions
\begin{enumerate}
\item $D$ is connected, and
\item if $x,y \in D$ then $f(x) f(y) \geq 0$.
\end{enumerate}
We say that $D$ is a {\em strong nodal domain} if (ii) is replaced by
\begin{enumerate}
\item[2'.] if $x,y \in D$, then $f(x) f(y) > 0$.
\end{enumerate}
\end{definition}

\begin{figure}
\centerline{
\begin{tabular}{c c c}
$d=3$&$d=4$&$d=5$\\
\includegraphics[bb=0 0 2016 1555,width=5cm]{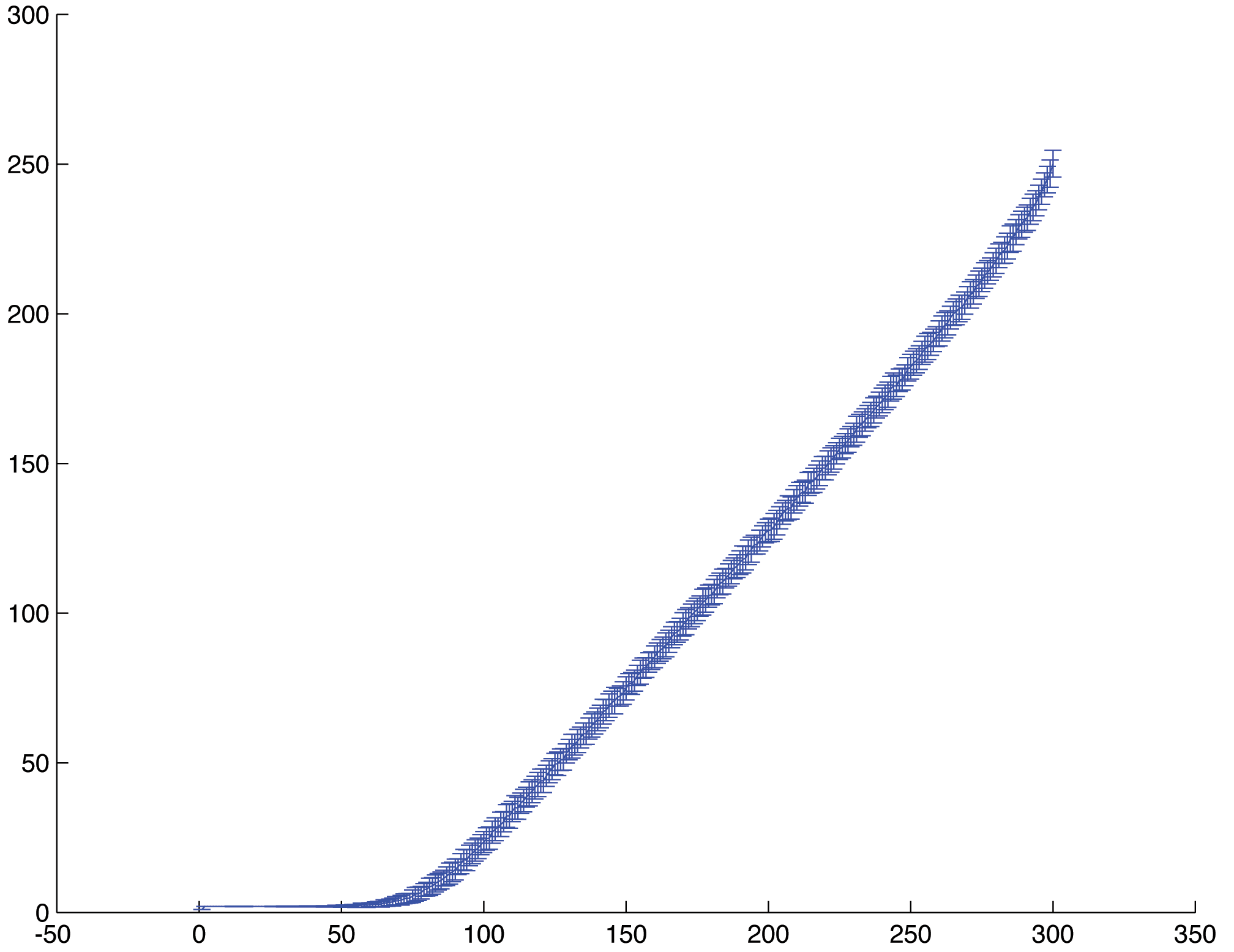} &
\includegraphics[bb=0 0 2007 1560,width=5cm]{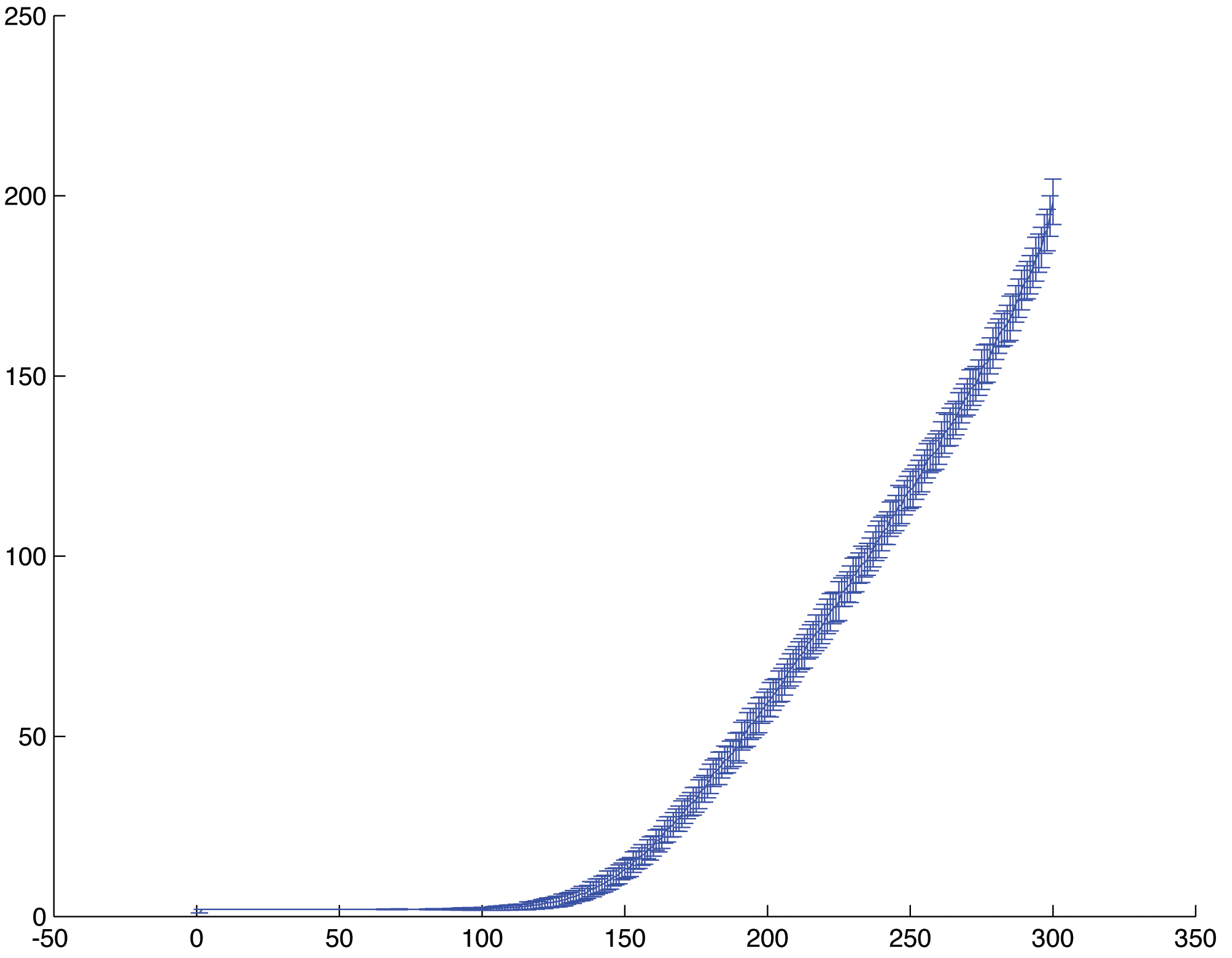} &
\includegraphics[bb=0 0 2025 1554,width=5cm]{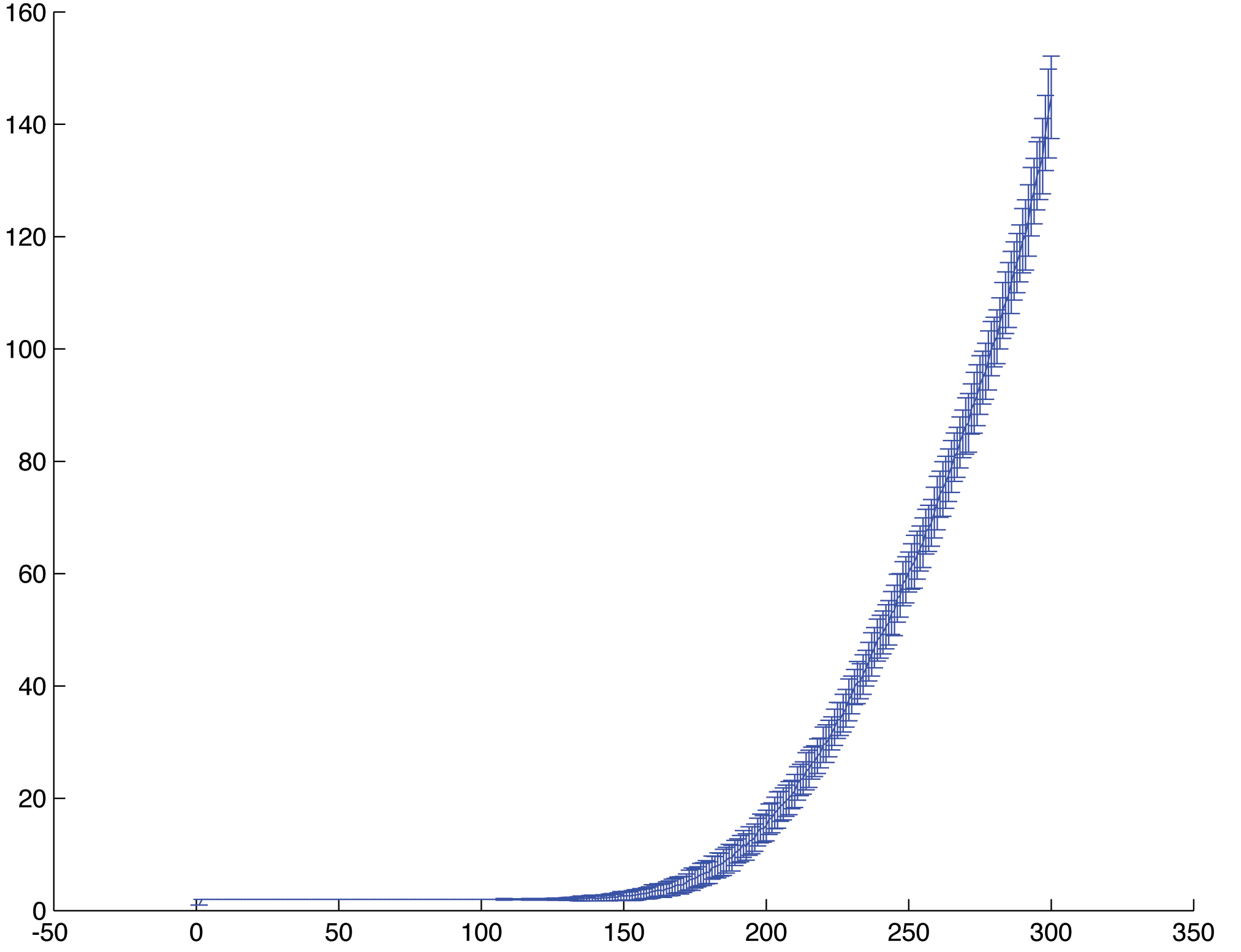}
\end{tabular}
}
\caption{\small The number of nodal domains in a random $d$-regular
$300$ vertex graph. There are $y$ nodal domains corresponding to the
$x$'th eigenvector (eigenvalues are sorted). For each $d$ we show the
average and standard deviation of $100$ such random graphs.}
\label{regular}
\end{figure}

\medskip

The main focus of our research is the following problem.

\begin{question}
 How many nodal domains (strong or weak) do the eigenvectors of $G$
 tend to have for $G$ that comes from various random graph models?
\end{question}

To get some initial idea, we started our research with a numerical
experiment whose outcomes were quite unexpected. We sampled
numerous graphs from the random graph space $G(n,\frac12)$.  It turned
out that in each and every one of these cases, {\em all} eigenvectors
of the adjacency matrix had {\em exactly two nodal domains}.  (In the experiments, no
eigenvector ever had a 0 coordinate, so the notions of strong and weak
domains are equivalent.)  The same experimental phenomenon was
observed for several smaller values of $p > 0$ in the random graph
model $G(n,p)$, provided that $n$ is large enough. Even more
unexpected were the results obtained for {\em random regular
 graphs}. Some of these results are shown in Fig.  \ref{regular}. We
found that quite a few of the first eigenvectors have just two nodal
domains, and only then the number of nodal domains starts to grow.

As mentioned above, there are other discrete analogs to the geometric
Laplacian. One often considers the so-called {\em combinatorial
 Laplacian} of a graph $G$. This is the matrix $D-A$ where $A$ is
$G$'s adjacency matrix and $D$ is a diagonal matrix with $D_{ii}$
being the degree of the vertex $v_i$.  It is well known that this
matrix is positive semidefinite, and it is also of interest to
investigate similar questions for the combinatorial Laplacian of
random graphs. The convention here is that the eigenvalues are sorted
as $0=\mu_1 \le \ldots \le \mu_n$. For regular graphs, the question
for the adjacency matrix and for the discrete Laplacian are
equivalent, since $\lambda_i + \mu_{n-i+1}=d$ for every $i$ and the
eigenfunction corresponding to $\lambda_i$ and to $\mu_{n-i+1}$ in the
two matrices are identical. However, for graphs in $G(n,p)$ it turns
out that the two spectra behave slightly differently. Computer
simulations suggest that all eigenvectors of the combinatorial
Laplacian that correspond to $\mu_2 \ldots \mu_{n-\Delta}$ for some
small $\Delta$ have exactly two nodal domains. However, among the last
$\Delta$ eigenfunctions, some have three nodal
domains. Fig. \ref{laplacian} suggests that in a constant fraction of
the graphs in $G \left( n, \frac12 \right)$ the eigenvector
corresponding to $\mu_n$ has three nodal domains.

\begin{figure}
\centerline{
\begin{tabular}{c}
\includegraphics[bb=0 0 512 385,width=6cm]{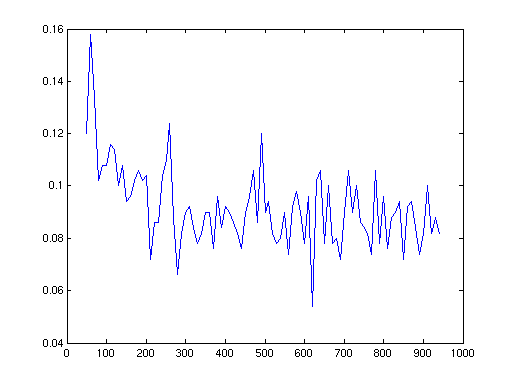} \\[-0.3cm]
{\footnotesize number of vertices}
\end{tabular}
}
\caption{\small The probability in $G(n,\frac12)$ that the last
eigenvector of the Laplacian has three nodal domains. For each $n$
there were $500$ experiments carried out.}
\label{laplacian}
\end{figure}

\medskip

Our main theorem for $G(n,p)$ partly establishes these observed
phenomena.

\begin{theorem}
\label{main_theorem}
For every $p \in (0,1)$, if $G \sim G(n,p)$, then asymptotically
almost surely the following holds for every eigenvector of $G$.  The
two largest weak nodal domains cover all vertices in $G$ with the
exception of at most $O_p(1)$ vertices, where $O_p(1)$ represents a
constant depending only on $p$.  The two largest strong nodal domains
cover all the vertices $\{ v \in V : f(v) \neq 0 \}$, with the
exception of at most $O_p(1)$ vertices.  In particular, in both cases
every eigenvector of $G$ has at most $O_p(1)$ (strong or weak) nodal
domains.
\end{theorem}

We remark that our bounds are quite reasonable (see Appendix A).
For instance, for $p=1/2$, we show that there are at most 46
exceptional vertices almost surely.

\subsection{Overview of our approach}

%Our basic approach is fairly straightforward, although it requires a subtle reduction
%to some nontrivial machinery.
Given $G \sim G(n,p)$ and an eigenvector $f$ of $G$, we first
partition $V(G) = \mathcal P_f \cup \mathcal N_f \cup \mathcal E_f$
where $\mathcal P_f, \mathcal N_f$ are the largest positive and
negative nodal domains of $f$, respectively, and $\mathcal E_f$ is a
set of exceptional vertices.  All eigenvectors are normalized to be
unit vectors in $\ell_2$.

We show that if $\mathcal E_f$ is large, then we can use the
eigenvalue condition, combined with upper estimates on the eigenvalues
of $G(n,p)$ to find a large subset $S$ of coordinates for which
$\|f|_S\|_2$ is smaller than one would expect for a random unit vector
$f \in S^{n-1}$.  The last step is to show that the probability that
some eigenvector has small 2-norm on a large set of coordinates is
exponentially small, and then apply a union bound over all such
subsets.

The problem is that the final step seems to require very strong upper
bounds on $\|A^{-1}\|$ (i.e. lower bounds on the smallest singular
value of $A$) for a random discrete matrix $A$.  Although there has
been a great deal of progress in this direction \cite{LPRTV05, LPRT05,
 Rudelson, TaoVu, TaoVuCost, RV07} (see also the survey
\cite{VuSurvey}), the known bounds for symmetric matrices are far too
weak for our purpose.  Thus it is crucial that we reduce to proving
upper bounds on $\|A^{-1}\|$ when $A$ is a random {\em rectangular}
$\pm 1$ Bernoulli matrix.  In this case, we can employ the optimal
bounds of \cite{LPRT05,LPRTV05} which also yield the exponential
failure probability that we require.  This reduction is carried out in
Theorem \ref{thm:l2mass}.

In Section \ref{sec:exactly}, we show that it is possible to get
significantly better control on the nodal domains of $G(n,p)$ if we
could get slightly non-trivial bounds on the $\ell_\infty$ norms of
eigenvectors.  This is a possible avenue for future research.

\subsection{Preliminaries}

\begin{notation}
 We denote by $G(n,p)$ a random graph with $n$ vertices, where each
 edge is chosen independently with probability $p$. The set of
 neighbors of a vertex $x \in V(G)$ is denoted by $\Gamma (x)$.
\end{notation}
\begin{notation}
 For a graph $G=(V,E)$ a function $f : V \to \mathbb R$, and any
 subset $S \subseteq V$, we denote $f(S) = \sum_{y \in S} f(y)$. In
 particular, for every $x \in V$ we denote $f(\Gamma (x)) = \sum_{y
   \sim x} f(y) = \sum_{y \in V} A_{xy} f(y)$, where $A$ is the
 adjacency matrix of $G$.
\end{notation}

As usual, the inner product of $f,g : V \to \mathbb R$ is denoted
$\langle f,g \rangle = \sum_{x \in V} f(x)g(x)$.  For $p \geq 1$, we
define $\|f\|_p = \left(\sum_{x \in V} |f(x)|^p\right)^{1/p}$, and
$\|f\|_\infty = \max_{x \in V} |f(x)|$.

%% \begin{notation}
%% Denote the eigenvalues of the adjacency matrix of a graph by
%% $\lambda_1 \geq \cdots \geq \lambda_n$. Unless stated otherwise, all
%% eigenvectors are normalized in $\ell_2$.
%% \end{notation}
\begin{definition}
 For $p \in [0,1]$, we define the random variable $X_p$ by
$$
X_p = \begin{cases}
 p-1 & \textrm{with probability $p$} \\
 p & \textrm{with probability $1-p$}.
\end{cases}
$$
In particular, $\mathbb E X_p = 0$.
\end{definition}
\begin{definition}
 Let $M_{m \times k}(p)$ be the $m \times k$ matrix whose entries are
 independent copies of $X_p$, and let $M_k^{\sym}(p)$ be the
 symmetric $k \times k$ matrix whose entries above the diagonal are
 independent copies of $X_p$, and whose diagonal entries are $p$.
\end{definition}

Unless otherwise stated, throughout the manuscript, {\em all
 eigenvectors are assumed to be normalized in $\ell_2$.}  When $A$ is
the adjacency matrix of a graph, we arrange its eigenvalues as
$\lambda_1 \geq \lambda_2 \geq \ldots \geq \lambda_n$.  The
eigenvector corresponding to $\lambda_1$ is called the {\em first
 eigenvector,} and any other is a {\em non-first eigenvector.}

As is customary in this area, we occasionally say ``almost surely'' as
shorthand for ``asymptotically almost surely.''  All asymptotic
statements are made for large $n$ and it is always implicitly assumed
that $n$ is large enough.

\section{Spectral properties of random matrices}

We now review some relevant properties of random matrices.

\begin{theorem}[Tail bound for symmetric matrices]\label{thm:symtail}
If $A \sim M_k^{\sym}(p)$, then for every $\xi > 0$
\[
\Pr \left( \| A \| \geq (2 \sqrt{p(1-p)} + \xi ) \sqrt{k} \right) \leq
4 e^{ -(1 - o(1)) \xi^2 k / 8}
\]
Here $\| A \|$ stands for the $\ell_2$ operator norm of $A$.
\end{theorem}
\begin{proof}
%???? it would be good to mention the specific theorem we quote from either paper  ????
 F{\"u}redi and Koml{\'o}s prove in \cite[\S 3.3]{FK} that the
 expected value of the largest magnitude eigenvalue of $A$ is at most
 $(2+o(1)) \sqrt{p(1-p)} \sqrt{k}$. Alon, Krivelevich and Vu prove in
 \cite{AKV} (see also \cite{KV02}) that the probability that the
 largest eigenvalue of $A$ exceeds its median by $\xi \sqrt{k}$ is at
 most $2 e^{-\xi^2 k / 8}$, and so is the probability that the
 smallest eigenvalue of $A$ is smaller than its median by $\xi
 \sqrt{k}$.  As usual in the context of sharp concentration, the
 expected value of the first and last eigenvalues differs from the
 median by at most $O(1)$. The conclusion follows.
\end{proof}
We also need a similar bound for $A \sim M_{m \times k}(p)$ whose
proof is standard; we repeat it below
in order to record the exact dependence of the constants on $p$.

\begin{theorem}[Tail bound for non-symmetric matrices]\label{thm:nonsymtail}
 For any $m \geq k$, if $A \sim M_{m \times k}(p)$ then
 \[
 \Pr \left(\| A \| \geq a_1 \sqrt{p (1-p)} \sqrt{m} \right) \leq
 e^{-a_2 m}
 \]
 where $a_1,a_2$ are constants that depend only on $p$.
\end{theorem}

We first recall that $X_p$ has a subgaussian tail.
\begin{lemma}\label{lem:subgaussiantail}
 For every $0 \leq p \leq \frac12$,
 \[
 \mathbb E e^{t X_p} \leq e^{(1-p)^2 t^2 / 2}~.
 \]
For every $\frac12 < p \leq 1$,
\[
 \mathbb E e^{t X_p} \leq e^{p^2 t^2 / 2}~.
\]
\end{lemma}
\begin{proof}
 Note that
\[
   \mathbb E e^{t X_p}  =  p e^{-t (1-p)} + (1-p) e^{t p}.
\]
In both cases, the claim follows by comparing the Taylor
 series on both sides.
\end{proof}

\begin{proof}[Proof of \thmref{thm:nonsymtail}]
Let $N_1 \subseteq S^{k-1}$ and $N_2 \subseteq S^{m-1}$ be
$\frac13$-nets.
By standard estimates, $|N_1| \leq 9^k, |N_2| \leq 9^m$.

Fix
some $x\in N_1$ and $y\in N_2$, and estimate
$\Pr \left( \langle y,A x \rangle > t \sqrt{m} \right)$
as follows
 \[
 \Pr \left(  \langle y,A x \rangle  > t \sqrt{m} \right) = \Pr
 \left( e^{\lambda  \langle y,A x \rangle } > e^{\lambda t
     \sqrt{m}} \right) \leq \frac{\mathbb{E} e^{\lambda \langle y,A
     x \rangle}}{e^{\lambda t \sqrt{m}}}
 \]
Define $q = \max \left\{ p, 1-p \right\}$. By \lemref{lem:subgaussiantail} we have
 \begin{eqnarray*}
   \mathbb{E} e^{\lambda \langle y,A x \rangle } & = & \prod_{i=1}^k \prod_{j=1}^m
   \mathbb{E} e^{\lambda A_{ij} x_i y_j} \\
   & \leq & \prod_{i=1}^k \prod_{j=1}^m e^{q^2\lambda^2 x_i^2 y_j^2 / 2} = e^{q^2\lambda^2 /
     2}~.
 \end{eqnarray*}
 The optimal choice is $\lambda = \frac{t \sqrt{m}}{q^2}$
 which yields
 \[
 \Pr \left( \langle y,A x \rangle > t \sqrt{m} \right) \leq
 e^{-\frac{t^2 m}{2q^2}}
 \]
 and therefore
 \[
 \Pr \left( \exists x \in N_1, y \in N_2: | \langle y,A x \rangle | >
   t \sqrt{m} \right) \leq 2e^{-\frac{t^2 m}{2q^2}} | N_1 | | N_2
 | \leq 2e^{-\frac{t^2 m}{2q^2}} 9^{k+m}.
 \]

By successive approximation, express
$x\in S^{k-1}$ as $x = \sum_{i \geq 0} \alpha_i x_i$
where for all $i$, $x_i \in N_1$ and $|\alpha_i| \leq 3^{-i}$. Likewise
for $y\in S^{m-1}$ and $y = \sum_{i \geq 0} \beta_i y_i$ with $y_i \in N_2$
and $|\beta_i| \leq 3^{-i}$.
Therefore, if $| \langle y_i, Ax_j \rangle | \leq t \sqrt{m}$ for all $x_j \in N_1$ and $y_i
 \in N_2$, then
 \[
 | \langle y, Ax \rangle | \leq \sum_{i,j \geq 0} 3^{-i-j}
 | \langle y_i, Ax_i \rangle | \leq 3 t \sqrt{m},
 \]
 which means that
 \[
 \Pr \left( \| A \| \geq 3 t \sqrt{m}\right) \leq 9^{m+k} e^{-\frac{ t^2 m}{2q^2}}~.
 \]
%We want this probability to be exponentially small in $m$, so
Now select $t$ to satisfy
$\frac{2 + \delta}{1 + \delta} \ln 9 < \frac{ t^2}{2q^2}$,
where $m = (1+\delta)k$. Concretely, let
$$t = 2 q \sqrt{2 \cdot \frac{2+\delta}{1 + \delta}
\ln 9}.$$
With this choice, we can take
 \begin{equation*}
 a_1 = \frac{18q}{4}
\sqrt{2 \cdot \frac{2+\delta}{1 + \delta} \ln 9}
 \end{equation*}
 \begin{equation*}\label{eq:a2}
 a_2 = 3
\frac{2+\delta}{1 + \delta} \ln 9.
 \end{equation*}
\end{proof}

\begin{theorem}[Tail bound for eigenvalues of $G(n,p)$]\label{thm:furkom}
 Let $G$ be a graph from $G(n,p)$ with $p \in (0, 1)$
 and let $\lambda_1, \ldots, \lambda_n$ be the eigenvalues of $G$'s adjacency matrix.
 Then for every $i \geq 2$ and every $\xi > 0$
 \[
 \Pr \left( | \lambda_i | \geq (2 \sqrt{p(1-p)} + \xi) \sqrt{n}
 \right) \leq \exp (-\xi^2 n / 32)
 \]
\end{theorem}
\begin{proof}
 F{\"u}redi and Koml{\'o}s prove in \cite[\S 3.3]{FK} that $\mathbb E
 \left( \max_{i\geq 2} | \lambda_i | \right) \leq 2 \sqrt{p(1-p)}
 \sqrt{n} (1 + o(1))$. As before, the theorem is derived by using a tail
 bound from \cite{AKV}.
\end{proof}

We now state the following theorem from \cite[Thm 3.3]{LPRTV05} (a
slight generalization of \cite{LPRT05}), specialized to the case of
the sub-gaussian random variables $X_p$.

\begin{theorem}\cite{LPRTV05}\label{thm:LPRTV}
  For any $p \in (0, 1)$, $\delta > 0$ there exist constants $\alpha =
  \alpha(p,\delta) > 0, \beta = \beta(p) > 0$ such that the
  following holds for all sufficiently large $k$ and every $w \in
  \mathbb R^m$:
$$
\Pr\left[\exists v \in S^{k-1} \textrm{ s.t. } \|Qv - w\|_2 \leq
  \alpha \sqrt{m}\right] \leq \exp(-\beta m)
$$
where $m=(1+\delta)k$ and the probability is taken over $Q \sim M_{m
  \times k}(p)$.
\end{theorem}

Strictly speaking the results in \cite{LPRTV05} only apply for
symmetric random variables (and thus only to $X_{1/2}$ in our
setting).  Note, however, that the proof of the stated bound in \cite{LPRTV05} is
based on three ingredients: (i) The exponential tail bound on the
operator norm from Theorem \ref{thm:nonsymtail}; (ii) A
small-ball probability estimate that is based on the Paley-Zygmund
inequality; and (iii) A small-ball probability estimate, based on
a Berry-Ess\'een-type inequality. We note that step (iii)
does not require a symmetry assumption (see, e.g. \cite[\S 2.1]{Stroock93}).

In order to derive explicit bounds on $\alpha$ and $\beta$, we
now give a simpler proof of \cite[Thm 3.3]{LPRTV05} (which proof can essentially
be read off from \cite{LPRTV05}) that works
when $p$ and $\delta$ are large enough, and uses only (ii).
For instance, we get non-trivial results for random graphs
in $G(n,p)$ when $p \in [0.18, 0.78]$.  In general,
this does not require $\delta$ to be small; for instance,
the simpler proof yields positive values for $\alpha$ and
$\beta$ when $p=1/2$ and $1+\delta \geq 1000$,
which is sufficient for the applications in
Section \ref{sec:domains}.

These estimates are only used in the appendix
for calculating explicit upper bounds on the
number of exceptional vertices, whereas in the rest
of the paper we use the statement
of Theorem \ref{thm:LPRTV} above.  Nevertheless,
the proof is fairly straightforward,
and carries pedagogical value for readers
not familiar with such techniques.
First, we recall the Payley-Zygmund inequality \cite{PZ32}.
%The proof of (2) for the random variable $X_p$ utilizes several known
%facts which we now recall. The first is Paley-Zygmund Inequality
%\cite{PZ32}.

\begin{lemma}\label{lem:PZ}
  For any positive random variable $Z$ and any $0 \leq \lambda \leq 1$
  \[
  \Pr \left( Z \geq \lambda \mathbb E (Z) \right) \geq (1 - \lambda)^2
  \frac{\left(\mathbb E Z\right)^2}{\mathbb E (Z^2)}.
  \]
\end{lemma}

\remove{
The following inequality appears in \cite[\S 6.3]{LT}:

\begin{lemma}[Standard symmetrization
  inequality]\label{lem:symmetrization}
  Let $F:\mathbb R_+ \rightarrow \mathbb R_+$ be a convex function,
  $\varepsilon_1, \ldots, \varepsilon_n$ independent Rademacher random
  variables and $X_1, \ldots, X_n$ independent random variables with
  mean zero. Then
  \[
  \mathbb E F \left( \left| \sum_i X_i \right| \right) \leq \mathbb E
  F \left( 2 \left| \sum_i \varepsilon_i X_i \right| \right)~.
  \]
\end{lemma}
We next recall Khinchine's inequality (e.g., \cite{Young76}):
\begin{lemma}[Khinchine's inequality]\label{lem:khinchine}
  Let $\varepsilon_1, \ldots, \varepsilon_n$ be independent Rademacher
  random variables. Let $0 < p < \infty$ and $x_1, \ldots, x_n \in
  \mathbb C$. Then
  \[
  \left( \mathbb E \left| \sum_i \varepsilon_i x_i \right|^q
  \right)^{1/q} \leq B_q \left( \sum_i |x_i|^2 \right)^{1/2}
  \]
  where $B_q$ is a constant that depends only on $q$.
\end{lemma}
}

We can now prove (ii), specialized to the random variables $X_p$.
\begin{lemma}\label{lem:paleyzygmund}
  Let $Y_1, \ldots, Y_n$ be independent copies of $X_p$ and let $a_1,
  \ldots, a_n \in \mathbb R$ satisfy $\sum a_i^2 = 1$, then for
  every $s \in \mathbb R$ and every $0 \leq \eta \leq 1$
  \[
  \Pr \left( \left| \sum_{i=1}^n a_i Y_i - s \right| > \eta \sqrt{
      p(1-p) } \right) > C (1 - \eta^2)^2
  \]
  where $C = \frac{p^2(1-p)^2}{128q^4}$ and $q = \max \left\{ p, 1-p
  \right\}$.
\end{lemma}
\begin{proof}

Let $X = \sum_{i=1}^n a_i Y_i - s$.  By independence, one has
$$
\mathbb E[X^2] = \sum_{i=1}^n a_i^2 \mathbb E(Y_i^2) + s^2 = p(1-p) + s^2.
$$
Also, setting $Y = \sum_{i=1}^n a_i Y_i = X+s$,
we have
\begin{eqnarray*}
\mathbb E[Y^4] &=& 4 \int_{0}^\infty t^3 \Pr\left(\left| \sum_{i=1}^n a_i Y_i \right| > t\right)\,dt \\
&\leq& 8 \int_0^{\infty} t^3 e^{-t^2/2q^2} \,dt = 16q^4 .
\end{eqnarray*}
This inequality follows from the subgaussian tail bound in
\lemref{lem:subgaussiantail}. For $\Pr \left( \sum a_i Y_i > t
\right)$ we use the result for $X_p$ and for $\Pr \left( - \sum a_i
Y_i > t \right)$ we use the result for $-X_p =
X_{1-p}$. Here it is convenient to reformulate \lemref{lem:subgaussiantail}
by saying that $\mathbb E e^{tX_p} \leq e^{q^2 t^2 / 2}$ for all $0 \le p \le 1$.

Using  $(a-b)^4 \le 8(a^4 + b^4)$, we have
$$
\mathbb E[X^4] \leq 8(\mathbb E[Y^4] + s^4) \leq 128q^4 + 8 s^4.
$$

  We can now apply \lemref{lem:PZ} with $Z = X^2$ and conclude that
  \[
  \Pr \left(X^2 \geq \lambda (p(1-p) + s^2) \right) \geq (1-\lambda)^2
  \frac{(p(1-p) + s^2)^2}{128q^4 + 8s^4}
  \]
  This last expression has a single local minimum at $s=0$ and tends to $1/8$
  as $s \rightarrow \infty$. Its value when $s=0$ is
  $\frac{p^2(1-p)^2}{128q^4}$ ($= C$) which is at most $\tfrac1{128} < \tfrac18$.

  Thus we have
  \[
  \Pr \left( X^2 \geq \lambda p(1-p) \right) \geq \Pr \left( X^2 \geq
    \lambda(p(1-p) + s^2) \right) \geq C (1-\lambda)^2
  \]

  Letting $\eta = \sqrt{\lambda}$ gives us the desired result.
\end{proof}

\remove{
\begin{lemma}
 For every $p > 0$, there exists a constant $C(p)$ such that the
 following holds.  Let $s \in \mathbb R$ and $a_1, \ldots, a_n \in
 \mathbb R$ be such that $\sum_{i=1}^n a_i^2 = 1$.  If $Y_1, \ldots,
 Y_n$ are i.i.d. copies of $X_p$ and if $Y = |\sum_{i=1}^n a_i Y_i -
 s|$, then
$$\mathbb E[Y^4] \leq C(p) (\mathbb E[Y^2])^2.$$
\end{lemma}
}

\remove{
\begin{proof}
 We have $\mathbb E[Y^4] = \mathbb E [|\sum_{i=1}^n a_i Y_i -s |^4]
 \leq s^4 + \mathbb E[|\sum_{i=1}^n a_i Y_i|^4]$.  Bound the latter
 term by
\begin{eqnarray*}
 \mathbb E\left[\left|\sum_{i=1}^n a_i Y_i\right|^4\right] &=&
 4 \int_{0}^{\infty} \lambda^3 \Pr\left(\left|\sum_{i=1}^n a_i Y_i\right| > \lambda\right)\,d\lambda \\
 &\leq &
 4 \int_0^{\infty} \lambda^3 (2\,e^{-\lambda^2/2})\,d\lambda = 16,
\end{eqnarray*}
where the second line follows immediately from Hoeffding's inequality.
Thus $\mathbb E[Y^4] \leq s^4 + 16$, while $\mathbb E[Y^2] = p(1-p) +
s^2$.
\end{proof}
}

We now present a proof of Theorem \ref{thm:LPRTV} which only
works for certain values of $p,\delta$ in order
to calculate explicit upper bounds in the appendix.

\begin{proof}[Proof of Theorem \ref{thm:LPRTV} holding
only for certain values of $p,\delta$]
  We are seeking an upper bound for
  \[
  \Pr \left( \exists x \in S^{k-1} \textrm{ s.t. } \|Ax - w \|_2 \leq
    \alpha \sqrt{m} \right)
  \]
  for $A \sim M_{m \times k}(p)$. We separately estimate this
  probability depending on whether $\| A \| > a_1 \sqrt{m}$ or not,
  where $a_1$ is the constant from Theorem \ref{thm:nonsymtail}. An
  upper bound on the probability that $\| A \| > a_1 \sqrt{m}$ is
  given in \thmref{thm:nonsymtail}. For the complementary case we use
  an $\gamma$-net $N$ on $S^{k-1}$. Let
  $x \in N$ and $w \in \mathbb R^m$. Let $f_i = | \sum_{j=1}^k A_{ij}
  x_j - w_i |$ for every $i \leq m$. By \lemref{lem:paleyzygmund} we
  have
 \[
 \Pr \left( f_i > \eta \right) > C \left( 1 - \tfrac{\eta^2}{p(1-p)}
 \right)^2~.
 \]
 Define $b = C \left(1 - \frac{\eta^2}{p(1-p)}\right)^2$.

 Let us bound the probability that $\| Ax - w \|^2 \leq c^2 m$ for
 some constant $c$ to be chosen later.
\begin{eqnarray*}
 \Pr \left( \| Ax - w \|^2 \leq c^2 m \right) &=& \Pr \left(
   \sum_{i=1}^m f_i^2 \leq c^2 m \right) \\
 &=& \Pr \left( m - \frac1{c^2} \sum_{i=1}^m f_i^2 \geq 0 \right) \\
 &=& \Pr
 \left( \exp \left( \tau m - \frac{\tau}{c^2} \sum_{i=1}^m f_i^2
   \right) \geq 1 \right) \\
 &\leq& \mathbb E \left( \exp \left( \tau m - \frac{\tau}{c^2}
     \sum_{i=1}^m f_i^2 \right) \right) = e^{\tau m} \prod_{i=1}^m
 \mathbb E \left( \exp \left( -\frac{\tau f_i^2}{c^2} \right) \right),
\end{eqnarray*}
 for every $\tau > 0$.

 For every $i \leq m$
 \begin{eqnarray*}
   \mathbb E \left( \exp \left( -\frac{\tau f_i^2}{c^2} \right) \right) &
   = & \int_0^1 \Pr \left( \exp \left( -\frac{\tau f_i^2}{c^2} \right) >
     t \right) dt \\
   & \leq & \int_0^{e^{-\tau \eta^2 / c^2}} dt + \int_{e^{-\tau \eta^2 /
       c^2}}^1 (1 - b) dt \\
   & = & e^{-\tau \eta^2 / c^2} + (1 - b)
   \left( 1 - e^{-\tau \eta^2 / c^2} \right) = 1 - b \left( 1 -
     e^{-\tau \eta^2 / c^2} \right)
 \end{eqnarray*}
 so we have
 \[
 \Pr \left( \| Ax - w \|^2 \leq c^2 m \right) \leq e^{\tau m} \left(
   1 - b \left( 1 - e^{-\tau \eta^2 / c^2} \right) \right)^m \leq
 e^{\tau m - b \left( 1 - e^{-\tau \eta^2 / c^2} \right)m}
 \]
 This expression is minimized for
 \[
 \tau = \frac{c^2}{\eta^2} \ln \frac{b \eta^2}{c^2}~.
 \]
 For $\tau$ to be positive (as it should), we must have $c < \sqrt{b}
 \eta = \sqrt{C} \left(1 - \frac{\eta^2}{p(1-p)} \right) \eta$.
 Letting $\eta =
 \sqrt{\frac{p(1-p)}3}$ gives
 \[
 b = \frac{4C}9~.
 \]
 Let
 \[
 c = (1 - \theta) \frac23 \sqrt{ \frac{p(1-p) C}3}
 \]
 for some $0 < \theta \leq 1$, which gives
 \[
 \tau = \frac49 (1-\theta)^2 C \ln
 \frac1{(1-\theta)^2}~.
 \]
 Using these values for $\eta, b, \tau$ and $c$ gives
 \[
 \Pr \left( \| Ax - w \|^2 \leq c^2 m \right) \leq \exp \left(
   -\frac{4C}9 \left[ 1 + (1-\theta)^2 \left( 2 \ln
       (1-\theta) - 1 \right) \right] m \right)
 \]

 The size of an $\gamma$-net on $S^{k-1}$ is at most $(1 + 2/
 \gamma )^k \leq \left( \frac3{\gamma} \right)^k$
 assuming $\gamma \leq 1$. Repeating this argument for every
 $x \in N$ and using the union bound we get that
\begin{eqnarray*}
 \Pr \left( \exists x \in N \textrm{ s.t. } \|Ax - w \|_2 \leq c
   \sqrt{m} \right) &\leq& \exp \left( -\frac{4C}9 \left[ 1 +
     (1-\theta)^2 \left( 2 \ln (1-\theta) - 1 \right)
   \right] m + k \ln \frac3{\gamma} \right) \\
 &=& \exp \left( - \left\{ \frac{4C}9 \left[ 1 + (1-\theta)^2
       \left( 2 \ln (1-\theta) - 1 \right) \right] - \frac{\ln
       \frac3{\gamma}}{1+\delta} \right\} m \right)
\end{eqnarray*}
 Now we seek a constant $\alpha$ such that if (i) $\| A \| \le a_1
 \sqrt{m}$, and (ii) there exists $x \in S^{k-1}$ with $\| Ax - w \|_2
 \leq \alpha \sqrt{m}$, then $\exists x' \in N$ with $\| Ax' - w \|_2
 \leq c \sqrt{m}$. If $x' \in N$ is chosen so that $\| x - x' \|_2
 \leq \gamma$, then
 \[
 \| Ax' - w \|_2 = \| Ax - w + A(x' - x) \|_2 \leq \| Ax - w \|_2 + \|
 A(x' - x) \|_2 \leq \alpha \sqrt{m} + \gamma a_1 \sqrt{m}~.
 \]

 By Theorem \ref{thm:nonsymtail} we have $\Pr ( \| A \| > a_1
 \sqrt{m}) \leq e^{-a_2 m}$, so taking
 \begin{equation}\label{eq:alpha}
   \alpha = (1 - \theta) \frac23 \sqrt{ \frac{p(1-p)C}3} -
   \gamma a_1
 \end{equation}
 and
 \begin{equation}\label{eq:beta}
   \beta = \min \left( a_2, \frac{4C}9 \left[ 1 + (1-\theta)^2
       \left( 2 \ln (1-\theta) - 1 \right) \right] - \frac{\ln
       \frac3{\gamma}}{1+\delta} \right)
 \end{equation}
 gives us the desired result.
\end{proof}

This preceding argument yields the conclusion
of Theorem \ref{thm:LPRTV} as long as there exists
a choice of $\gamma,\theta \in (0,1)$ such that
both $\alpha$ and $\beta$ are positive.
These bounds are used in Appendix A to
show that our techniques, if analyzed
in gory detail, yield reasonable
bounds on the number of nodal domains
for various values of $p$.

\medskip

The next result follows from taking a union bound in Theorem
\ref{thm:LPRTV}.

\begin{corollary}\label{cor:LPRTV}
 Maintaining the notations of Theorem \ref{thm:LPRTV}, it is moreover
 true that
$$\Pr\left[\exists c \in \mathbb R, \exists v \in
 S^{k-1} \textrm{ s.t. } \|Qv - cw\|_2 \leq \alpha \sqrt{m}\right]
\leq \exp(-\beta m)$$ for all sufficiently large $k$, and every $w \in
\mathbb R^m$.
\end{corollary}
\begin{proof}
  We may assume that $\|w\|_2 = 1$.  Since clearly $\|Qv\|_2 \leq m$,
  it suffices to prove the bound for $c \in [-2m,2m]$.  Apply Theorem
  \ref{thm:LPRTV} for every integer $j \in [-2m,2m]$ and a union bound
  to conclude that
  \[
  \Pr \left( \exists \textrm{ an integer } j \in [-2m,2m] \textrm{ and } v
    \in S^{k-1} \textrm{ s.t. } \| Qv - j w \|_2 \leq \alpha \sqrt{m}
  \right) \leq \exp \left( -\beta m + \ln \left( 4m+1 \right) \right)
  \]
  But now for any $c \in [-2m,2m]$, let $j \in [-2m,2m]$ be the
  nearest integer.  For any matrix $Q$ and $v \in S^{k-1}$
  \[
  \| Qv - c w \|_2 \leq \| Qv - jw \|_2 + \| \left( j - c \right) w
  \|_2 \leq \| Qv - jw\|_2 + 1~.
  \]
\end{proof}

\section{Nodal domains}
\label{sec:domains}

\subsection{Eigenvectors are not too localized}

We show first that the restriction of an eigenvector to a large
set of vertices must have a substantial $\ell_2$ norm.

\begin{theorem}[Large mass on large subsets]
\label{thm:l2mass}
For every $p \in (0, 1)$ and every $\varepsilon > 0$, there
exist $\alpha = \alpha(\varepsilon,p) > 0$ and $\beta = \beta(p) > 0$ such that
for $n$ large enough, and for every {\em fixed} subset
$S \subseteq [n]$ of size $|S| \geq (\frac12+\varepsilon)n$,
$$
\Pr\left[ \exists \textrm{a non-first eigenvector $f$ of $G$ satisfying
$\|f|_S\|_2 < \alpha$}\right] \leq \exp(-\beta n)
$$
where the probability is over the choice of $G \sim G(n,p)$.
%???? are we able to say anything quantitative here on how $\alpha$
%and $\beta$ depend on $p$  ????
% Not without fishing around a lot in the LPRTV proofs.
% I think it's probably not that worth it...
\end{theorem}

\begin{proof}
Let $A$ be the adjacency matrix of $G = (V,E)$, and let $f : V \to \mathbb R$
be a non-first eigenvector of $G$ with eigenvalue $\lambda$.
Assume w.l.o.g. that $\alpha \leq \frac12$, and let
$S \subseteq V$ be as in the theorem.

For every $x \in S$, the eigenvector condition
$\lambda f(x) = f(\Gamma(x))$ implies that

$$
\left| \sum_{y \in V} A_{xy} f(y) \right| = |\lambda f(x)|.
$$

Or equivalently,
$$
\left| \sum_{y \in V} (p-A_{xy})f(y) -p  \sum_{y \in V} f(y) \right|
= |\lambda f(x)|.
$$
Squaring and summing over all $x \in S$ this yields
\begin{equation}\label{eq:smallsum}
\sum_{x \in S} \left|\sum_{y \in V} (p-A_{xy}) f(y) - p\sum_{y \in
V} f(y) \right|^2 = |\lambda |^2 \cdot \|f|_S\|^2_2.
\end{equation}
Let us define $M = pJ - A$, where $J$ is the $n\times n$ all ones matrix,
and let $B$ be the $|S| \times n$ sub-matrix of $M$ consisting of rows corresponding to vertices
in $S$, then \eqref{eq:smallsum} implies that $\|B f - p\bar f \1\|_2 = |\lambda| \cdot \|f|_S\|_2$,
where $\bar f = \langle f,\1 \rangle$.

Furthermore, if we decompose $B = [P~~Q]$ where $P$ contains the
columns corresponding to vertices in $S$, and $Q$ the others, then
clearly $P \sim M^{\mathrm{sym}}_{|S|}(p)$ and $Q \sim M_{|S| \times
(n-|S|)}(p)$. (It is useful to think of $P$ as corresponding
to the subgraph of $G$ induced by $S$, whereas $Q$ corresponds to
the bipartite graph corresponding to $S$ and its complement).
Since $Bf = P(f|_{S}) + Q(f|_{\overline S})$,
%(should be $Bf = P(f|_{S}) + Q(f|_{\overline S})$),
we can write
\begin{equation}
\label{eq:Q}
\|Q(f|_{\overline S}) - p\bar f \1 \|_2 \leq \|Bf - p\bar f \1\|_2 + \|P f|_S\|_2
= |\lambda| \cdot \|f|_S\|_2 + \|P(f|_S)\|_2
\end{equation}
The main point now is that, by Corollary \ref{cor:LPRTV}, $Q$ (a random rectangular matrix
with {\em all} entries being i.i.d. copies of $X_p$) is extremely unlikely
to map $f|_{\overline S}$ near the 1-dimensional subspace spanned by $\mathbf{1}$.
But we know from concentration results that if $f|_S$ has small 2-norm, then the
RHS will be small.  This leads to a contradiction.

Now, since $f$ is a non-first
eigenvector, by Theorems \ref{thm:furkom} and \ref{thm:symtail},
there exist constants $C = C(p), \beta' = \beta'(p) > 0$ such that
\begin{equation}\label{eq:betaprime}
\Pr\left[|\lambda| + \|P\| \geq C \sqrt{n}\right] \leq \exp(-\beta' n).
\end{equation}

If we assume that $|\lambda| + \|P\| \leq C \sqrt{n}$ and also
$\|f|_S\|_2 < \alpha \leq \frac12$, then
$a = \|f|_{\overline S}\|_2 = \sqrt{1 - \alpha^2} > \frac12$.
In this case, \eqref{eq:Q} implies that
\begin{equation}\label{eq:fbound}
\|Q(\tfrac{1}{a} f|_{\overline S}) - \tfrac{p}{a} \bar f \1\|_2 \leq 2C\sqrt{n}\alpha.
\end{equation}

Now let $k = n-|S|$. By Corollary \ref{cor:LPRTV},
for any $\delta \geq \frac{4\varepsilon}{1-2\varepsilon} > 0$,
there exists $\alpha,\beta > 0$ (with $\beta$ depending
only on $p$) such that for $n$ large enough,
\begin{equation}\label{eq:useLPRT}
\Pr_{Q \sim M_{k(1+\delta) \times k}(p)} \left[\exists v \in
S^{k-1}, \exists c \in \mathbb R \textrm{ s.t. } \|Q v - c
\1\|_2 \leq 2C\alpha\sqrt{n}\right] \leq \exp(-\beta n),
\end{equation}
but this implies that \eqref{eq:fbound} and $|\lambda| + \|P\| \leq C\sqrt{n}$
occur with probability
at most $\exp(-\beta n)$.
Taking a union bound over \eqref{eq:useLPRT} and \eqref{eq:betaprime}, we see that for some $\alpha > 0$,
$$
\Pr[\|f|_S\|_2 < \alpha] \leq \exp(-\beta n) + \exp(-\beta' n),
$$
where both $\beta, \beta' > 0$ depend only on $p$.  This completes the proof.
\end{proof}

We record the following simple corollary.

\begin{corollary}\label{cor:l2mass}
For every $p \in (0, 1)$ there exist $r = r(p) > 0$ and
$\varepsilon = \varepsilon(p)$ with $0 < \varepsilon < \frac12$
such that for almost all $G \sim G(n,p)$
and {\em every} subset $S \subseteq V(G)$
with $|S| \geq (\frac12+\varepsilon) n$ we have $\|f|_S\|_2 \geq r$.
\end{corollary}

\begin{proof}
We need to take a union bound over all subsets $S \subseteq [n]$
with $|S| \geq (\frac12+\varepsilon) n$.  But for every value $\beta = \beta(p)$
from Theorem \ref{thm:l2mass}, there exists an $\varepsilon < \frac12$ such that
the number of such subsets is $o(\exp(\beta n))$, hence the union bound applies.
\end{proof}

\subsection{Bounding the number of exceptional vertices}

We now turn to prove Theorem \ref{main_theorem}.
Fix some $p \in (0, 1)$ and consider $G(V,E) \sim G(n,p)$.
We concentrate on the part of the theorem that
concerns weak nodal domains; exactly the same proof works for
strong domains after we delete the vertices at which an
eigenvector vanishes (which do not contribute to the
count of strong nodal domains).

Let $r = r(p)$ and $\varepsilon = \varepsilon(p)$ be chosen as in
Corollary \ref{cor:l2mass}. In everything that follows, we assume
that $n$ is sufficiently large. By Theorem \ref{thm:furkom},
it almost surely holds that every non-first eigenvalue $\lambda$ of $G$
satisfies $|\lambda| = O(\sqrt{p(1-p)n})$. We thus can and will
assume that this bound holds for the remainder of this section.

\begin{lemma}\label{lem:smallsum}
 Almost surely every non-first eigenvector $f : V \to \mathbb R$ of
 $G$ satisfies $|\langle f, \mathbf{1}\rangle| \leq O(1/\sqrt{p})$.
\end{lemma}

\begin{proof}
 As usual, $A$ is the adjacency matrix of $G$, $\lambda$ is the
 eigenvalue of $f$ and $J$ is the all-ones matrix. Let $M = pJ -
 A$. Then,
$$
\lambda \langle f, \mathbf{1} \rangle = \langle f, A\mathbf{1} \rangle
= \langle f, pJ\mathbf{1}\rangle - \langle f, M\mathbf{1}\rangle = pn
\langle f, \mathbf{1}\rangle - \langle f,M\mathbf{1}\rangle.
$$
It follows that
$$
|\langle f, \mathbf{1}\rangle| = \frac{|\langle f,
 M\mathbf{1}\rangle|}{pn-\lambda} \leq
\frac{\|M\mathbf{1}\|_2}{pn-\lambda} \leq \frac{\sqrt{n}
 \|M\|}{pn-\lambda}\,.
$$
But since $M \sim M^{\mathrm{sym}}_n(p)$, we know that $\|M\| =
O(\sqrt{p(1-p)n})$ almost surely.  Therefore almost surely we have
$$
|\langle f, \mathbf{1}\rangle| \leq \frac{O(n
 \sqrt{(1-p)})}{pn-O(\sqrt{(1-p)n})} \leq O\left(\frac{1}{\sqrt{p}}\right).
$$
\end{proof}

We now derive an almost sure bound on the number of weak nodal domains.
To this end, for any function $f : V \to \mathbb R$,
we let $\mathcal P_f$ and $\mathcal N_f$ be the largest
non-negative and non-positive domains in $f$ (with respect
to the random graph $G$).  Define $\mathcal E_f = V
\setminus (\mathcal P_f \cup \mathcal N_f)$, and let $\mathcal Z_f =
\{ v \in V : f(v) = 0 \}$.  In particular, observe that $\mathcal E_f
\cap \mathcal Z_f = \emptyset$ since we are discussing weak domains.
Although the next lemma is stated in terms of nodal domains, it should
be clear that it is, in fact, a simple combinatorial observation about
random graphs.

\begin{lemma}
 \label{lem:smallcomp}
For any $f : V \to \mathbb R$,
if $D_1, \ldots, D_m$ are the weak nodal domains in $\mathcal E_f$,
 then almost surely $m = O(p^{-1} \log n)$ and $|D_i| = O(p^{-1} \log
 n)$ for every $i \in [m]$.
\end{lemma}

\begin{proof}
 If $P_1, P_2, \ldots, P_s$ are non-negative nodal domains, then
 selecting one element from every $P_i$ yields an independent set of
 size $s$.  By a standard fact (which follows immediately from
  a union bound) about graphs in $G(n,p)$, almost
 surely $s = O(p^{-1} \log n)$.  Furthermore, since each $|P_i| \le
 |\mathcal P_f|$ and there are no edges between the two sets, almost
 surely $|P_i| \le O(p^{-1} \log n)$.  The same holds for the
 non-positive nodal domains in $\mathcal E_f$.
\end{proof}

We are now ready to complete the proof of Theorem~\ref{main_theorem}.
We will use the following straightforward fact about $G(n,p)$.

\begin{fact}
For any fixed $k \in \mathbb N$, the following holds almost surely
over the choice $G \sim G(n,p)$.  For any $x_1, x_2, \ldots, x_k \in V$,
$|\Gamma(x_1) \cup \cdots \cup \Gamma(x_k)| = (1-(1-p)^k \pm o(1)) n$
and $|\Gamma(x_1) \cap \cdots \cap \Gamma(x_k)| = (p^k \pm o(1))n.$
\end{fact}

\begin{lemma}\label{lem:fewsmall}
 It almost surely holds that every non-first eigenvector satisfies $|\mathcal
 E_f| = O_p(1)$.
\end{lemma}

\begin{proof}
 Since we seek a result that holds asymptotically almost surely, we
 take the following liberty: Any property that holds for almost every
 graph in $G(n,p)$ and is needed in the proof is assumed to hold for
 $G$.

 We can assume that $f$ has a constant sign on $\mathcal E_f$, for
 if $x,y \in \mathcal E_f$ satisfy $f(x) > 0 > f(y)$, then also
 $\Gamma(x) \cap \Gamma(y) \subseteq \mathcal E_f$.  But from Lemma
 \ref{lem:smallcomp}, $|\mathcal E_f| = O(\log^2 n)$, whereas
 $|\Gamma(x) \cap \Gamma(y)| = (p^2 + o(1))n$ by properties of random
 graphs.

 Let $k$ be an integer such that $\tfrac{\tfrac12 +
   \varepsilon}{1-(1-p)^k} < 1$. Let $\{x_1, x_2, \ldots, x_k\}
 \subseteq \mathcal E_f$ and assume without loss of generality that
 $f(x_i) < 0$ for $i = 1, \ldots, k$.  This implies that $|\mathcal
 P_f \cup \mathcal E_f| \geq (1-(1-p)^k-o(1))n$ (because this set
 contains the union of the neighborhoods of $x_1, x_2, \ldots, x_k$).
 Therefore $|\mathcal N_f \setminus \mathcal Z_f| \leq ((1-p)^k +
 o(1))n$.  By Lemma \ref{lem:smallsum}, we have
$$
\sum_{x \in \mathcal P_f \cup \mathcal E_f} f(x) \leq O(1/\sqrt{p}) + \sum_{x
 \in \mathcal N_f \setminus \mathcal Z_f} |f(x)| \leq O(1/\sqrt{p}) +
\sqrt{|\mathcal N_f \setminus \mathcal Z_f|} \sqrt{\sum_{x \in
   \mathcal N_f} f(x)^2} \leq \sqrt{((1-p)^k + o(1))n}.
$$

By Markov's inequality, we know that there exists a subset
$S \subseteq \mathcal P_f \cup \mathcal E_f$ such that\\
$|S| \geq \tfrac{\tfrac12 +
 \varepsilon}{1-(1-p)^k} |\mathcal P_f \cup \mathcal E_f|$, and for
every $y \in S$,
\begin{eqnarray*}
 f(y) &\leq& \frac{\sqrt{((1-p)^k+o(1))n}}{\left( 1 - \tfrac{\frac12+\varepsilon}{1-(1-p)^k} \right) |\mathcal P_f \cup \mathcal E_f|} \\
 &\leq & \frac{\left( 1 - (1-p)^k \right) \sqrt{((1-p)^k+o(1))n}}{(\frac12-\varepsilon-(1-p)^k) | \mathcal P_f \cup \mathcal E_f |} \\
 &\leq & \frac{\sqrt{((1-p)^k+o(1))}}{(\frac12-\varepsilon-(1-p)^k)\sqrt{n}}.
\end{eqnarray*}

\remove{
\begin{eqnarray*}
 f(y) &\leq& \frac{\sqrt{((1-p)^k+o(1))n}}{(\frac12-\varepsilon) |\mathcal P_f \cup \mathcal E_f|} \\
 &\leq & \frac{\sqrt{((1-p)^k+o(1))}}{(\frac12-\varepsilon)(1-(1-p)^k-o(1))\sqrt{n}}.
\end{eqnarray*}
}

Consequently,
\[
\| f |_S \|_2 \leq \sqrt{\mathcal P_f \cup \mathcal E_f|}
\frac{\sqrt{((1-p)^k+o(1))}}{(\frac12-\varepsilon - (1-p)^k)\sqrt{n}}
\leq \frac{\sqrt{((1-p)^k+o(1))}}{\frac12-\varepsilon - (1-p)^k}.
\]
\remove{
$$
\|f|_S\|_2 \leq \sqrt{|\mathcal P_f \cup \mathcal E_f|}
\frac{\sqrt{((1-p)^k+o(1))}}{(\frac12-\varepsilon)(1-(1-p)^k-o(1))\sqrt{n}}
\leq
\frac{\sqrt{((1-p)^k+o(1))}}{(\frac12-\varepsilon)(1-(1-p)^k-o(1))}.
$$
}
Corollary \ref{cor:l2mass} yields that $\|f|_S\|_2 \geq r = r(p)$.
It follows that
\begin{equation}\label{eq:krdependence}
 \frac{\sqrt{((1-p)^k+o(1))}}{\frac12-\varepsilon -(1-p)^k} \geq r.
\end{equation}
We conclude that $k = O\left(\frac{1}{p}
 \log\left(\frac{1}{r(\frac12-\varepsilon)}\right)\right)$, which finishes the
proof since $r, \varepsilon$ can be chosen depending only on $p$.
\end{proof}

\subsection{Future directions}
\label{sec:exactly}

First, we suspect that the following question should not be too
difficult to resolve.

\begin{conjecture}
For every fixed $p \in (0,1)$, for $G \sim G(n,p)$, almost surely every eigenfunction $f$ of $G$
  satisfies $$\{v \in V(G) : f(v) = 0\} = \emptyset.$$
\end{conjecture}

For simplicity in what follows, we only discuss weak nodal domains.
As before, $\mathcal P_f$ and $\mathcal N_f$ are the largest
non-negative and non-positive weak domains, and $\mathcal E_f = V
\setminus (\mathcal P_f \cup \mathcal N_f)$ is the set of exceptional
vertices.  We observe that if sufficiently good lower bounds on
$|\mathcal N_f|$ and $|\mathcal P_f|$ hold, then the number of
exceptional vertices is at most one. We illustrate this for $p =
\frac12$, but the extension to general $p$ is straightforward.

\begin{lemma}
 Suppose that there exists an $\epsilon_0 > 0$ such that almost surely,
 every non-first eigenvector $f$ of $G \sim G(n,\frac12)$, satisfies
 $|\mathcal P_f|, |\mathcal N_f| \geq (\frac14+\epsilon_0) n$. Then
 almost surely every eigenvector has at most one exceptional vertex.
\end{lemma}

\begin{proof}
 Almost surely, every pair of vertices $x,y \in V(G)$, satisfies
 $|\Gamma(x) \cup \Gamma(y)| \geq (\frac{3}{4}-o(1))n$.  Assuming the
 stated lower bound on $|\mathcal P_f|$ and $|\mathcal N_f|$, there
 must be an edge from $\{x,y\}$ to both $\mathcal P_f$ and $\mathcal
 N_f$.  But if $x$ and $y$ are both exceptional, then $f$ must have
 the same sign on $x$ and $y$ (see, e.g.  the proof of Lemma
 \ref{lem:fewsmall}), implying that there cannot be both types of
 edges.
\end{proof}

\medskip
\noindent
Next, we show that if one can obtain a slightly non-trivial upper
bound on $\|f\|_\infty$ for eigenvectors $f$ of $G(n,p)$, then almost
surely all eigenvectors of such graphs have at most $O(\frac{1}{p})$
exceptional vertices, and e.g. at most one exceptional vertex for $p
\in [0.21,0.5]$. First, we pose the following natural problem.

\begin{question}
 Is it true that, almost surely, every eigenvector $f$ of $G(n,p)$
 has $\|f\|_\infty = o(1)$?  The natural guess would be that for
 almost all graphs, every eigenfunction satisfies $\|f\|_\infty =
 n^{-\frac 12 + o(1)}$.
\end{question}

Very recently we learned that Tao and Vu \cite[Prop. 58]{TaoVuNew}
have made significant progress on this question by showing
that it holds for all but $o(n)$ of the eigenvectors.
The preceding question, though, is still open.
A positive answer yields more precise control on the
nodal domains of $G(n,p)$.

\begin{theorem}\label{thm:exactly2}
 Suppose that for almost every $G \sim G(n,p)$ it holds that all
 eigenvectors $f$ of $G$ satisfy $\|f\|_\infty = o(1)$. Then almost
 surely every eigenvector has at most $k_p = \lfloor
 \frac{1}{\log_2(1/(1-p))} \rfloor$ exceptional vertices.
\end{theorem}

In order to prove this statement, we need the following strengthening
of Theorem \ref{thm:l2mass}. In particular, we require that the subset
of vertices $S$ be allowed to (weakly) depend on the random choice of
$G \sim G(n,p)$.

\begin{theorem}
 \label{thm:l2mass2} For every $p \in (0, 1)$, and $\varepsilon > 0$,
 there exist values $\alpha = \alpha(\varepsilon,p) > 0$ and $\beta =
 \beta(p) > 0$ such that the following holds. Suppose $G \sim G(n,p)$
 and $A$ is the adjacency matrix of $G$.  Suppose further that $S
 \subseteq [n]$ is a (possibly) random subset which is allowed
 to depend on the rows of $A$
 indexed by a set $T \subseteq [n]$ with $|T| = o(n)$. Then for all
 sufficiently large $n$, we have
$$
\Pr\left[ \exists \textrm{a non-first eigenvector $f$ of $G$ with
   $\|f|_S\|_2 < \alpha \|f|_{V \setminus T}\|_2$ and $|S| \geq
   (\tfrac12+\varepsilon)n$}\right] \leq \exp(-\beta n).
$$
\end{theorem}

\begin{proof}
 Let $A$ be the adjacency matrix of $G = (V,E)$, and let $f : V \to
 \mathbb R$ be a non-first eigenvector of $G$ with eigenvalue
 $\lambda$. Let $S \subseteq V$ be a possibly random subset with $|S|
 \geq (\frac12+\varepsilon)n$. Let $T \subseteq [n]$ be such that
 $|T| \leq o(n)$ and $S$ is determined after conditioning on the
 values of $A$ in the rows indexed by $T$. Assume, furthermore, that
 $\|f|_S\|_2 < \alpha \|f|_{V \setminus T}\|_2$.

 Again, for every $x \in S$, the eigenvalue condition $\lambda f(x) =
 f(\Gamma(x))$ implies that
$$
\left| \sum_{y \in V \setminus T} A_{xy} f(y) + \sum_{y\in T} A_{xy}
 f(y) \right| = |\lambda f(x)|.
$$
Or equivalently,
$$
\left| \sum_{y \in V \setminus T} (p-A_{xy})f(y) -p \sum_{y \in V
   \setminus T} f(y) + \sum_{y \in T} A_{xy} f(y)\right| = |\lambda
f(x)|.
$$
Squaring and summing over all $x \in S \setminus T$ yields
\begin{equation}\label{eq:smallsum2}
 \sum_{x \in S \setminus T} \left(\sum_{y \in V \setminus T}
   (p-A_{xy}) f(y) - p\sum_{y \in V \setminus T} f(y) + \sum_{y \in T}
   A_{xy} f(y) \right)^2 = |\lambda |^2 \cdot \|f|_{S\setminus T}\|^2_2
\end{equation}
As above we define $M = pJ - A$, but now we let $B$ be the $|S
\setminus T| \times |V \setminus T|$ sub-matrix of $M$ consisting of
rows corresponding to vertices in $S\setminus T$ and columns
corresponding to vertices in $V \setminus T$. Now let $g =
\frac{f|_{V\setminus T}}{\|f|_{V \setminus T}\|_2}$, then
\eqref{eq:smallsum2} implies that $$\|B g - c_f w_T\|_2 \leq 2
|\lambda| \cdot \alpha$$ where $w_T \in \mathbb R^{|S\setminus T|}$ is
a vector which depends only on the rows of $A$ indexed by $T$ and $c_f
\in \mathbb R$ is some constant depending on $f$.  Note that we have
used the fact that $\|f|_S\|_2 \leq \alpha \|f|_{V\setminus
 T}\|_2$. Furthermore, $B$ and $S$ are independent random variables
conditioned on $T$. From this point, the proof proceeds just as in
Theorem \ref{thm:l2mass}.
\end{proof}

\begin{proof}[Proof of Theorem \ref{thm:exactly2}]
 Our goal is to show that almost surely $|\mathcal E_f| \leq k_p$ for
 every non-first eigenvector $f$ with associated eigenvalue
 $\lambda$. Suppose, to the contrary, that $|\mathcal E_f| > k_p$,
 and let $U \subseteq \mathcal E_f$ have $|U| = k_p+1$. Consider
 $\Gamma = \bigcup_{u \in U} \Gamma(u)$. By properties of random
 graphs, it holds that $$|\Gamma| \geq [1-(1-p)^{k_p+1}-o(1)]n \geq
 \left(\frac12+\varepsilon_p-o(1) \right)n,$$ for some positive $\varepsilon_p$.
 Thus for $n$ large enough, we may assume that indeed
 $|\Gamma| \geq \left(\frac12 + \varepsilon_p\right) n$.

 Again, by Theorem \ref{thm:furkom}, we have $|\lambda| =
 O_p(\sqrt{n})$ almost surely, and thus we assume this bound holds
 for the remainder of the proof.  Now for each $u \in U$, let $D_u$
 be the nodal domain of $u$ with respect to $f$.  Using the
 eigenvalue condition, for every $u \in U$, we have $|f(\Gamma(u))| =
 |\lambda f(u)| = O_p(\sqrt{n}) |f(u)|$.  Every neighborhood
 $\Gamma(u)$ has non-trivial intersection with only one of $\mathcal
 P_f$ or $\mathcal N_f$, hence grouping terms by sign, we have
\begin{eqnarray*}
 \sum_{x \in \Gamma(u)} |f(x)| &\leq & \left|f\left(\vphantom{\bigoplus}
     \Gamma(u) \cap (\mathcal P_f \cup \mathcal N_f)\right)
   + f\left(\vphantom{\bigoplus} \Gamma(u) \cap \mathcal E_f \setminus D_u \right)\right|
 + |f(D_u)| \\
 &\leq &
 |f(\Gamma(u))| + |D_u| \\
 &\leq &
 O_p(\sqrt{n}) |f(u)| + O(\log n),
\end{eqnarray*}
where we have used the estimate on $|D_u|$ from Lemma
\ref{lem:smallcomp} which holds almost surely.
%bound on $f$'s max norm.

In particular,
\begin{equation}\label{eq:single2}
 \sum_{u \in U} \sum_{x \in \Gamma(u)} |f(x)| \leq (k_p+1) \|f\|_\infty \cdot O_p(\sqrt{n}) + O(k_p \log n).
\end{equation}

Now let $\Gamma' = \{ x \in \Gamma : |f(x)| \leq \frac{c}{\sqrt{n}}\}$
for some $c = c(n) > 0$ to be chosen momentarily. Using
\eqref{eq:single2}, we see that $|\Gamma \setminus \Gamma'| \leq
\frac{O(k_p\sqrt{n}) |\lambda| \cdot \|f\|_\infty}{c}$. Under the
assumption $|\lambda| \cdot \|f\|_\infty = o(\sqrt{n})$, we can choose
$c = o(1)$ so that $|\Gamma \setminus \Gamma'| = o(n)$, in which case
we may assume that for $n$ large enough, $|\Gamma'| \geq \left(\frac12
 + \varepsilon'_p\right) n$ for some $\varepsilon'_p > 0$. We have
$\|f|_{\Gamma'}\|_2 = o(1)$, hence also
\begin{equation}\label{eq:normbound37}
 \|f|_{\Gamma'}\|_2 = o(1)
 \cdot \|f|_{V \setminus U}\|_2,
\end{equation}
since $\|f|_U\|_2 \leq \sqrt{|U|} \cdot \|f\|_\infty = o(1)$.

Since $\Gamma'$ depends on $f$ and not just on the rows of the adjacency matrix
corresponding to $U$, we cannot directly apply Theorem \ref{thm:l2mass2}.
Instead, we will apply the theorem to a collection of sets $\mathcal U$
which are determined
by $\Gamma$ (which {\em is} determined by the sets $\{\Gamma(u)\}_{u \in U}$),
with the guarantee that $\Gamma' \in \mathcal U$.

To this end, let $y(n) = o(n)$ be an upper bound on the size of $|\Gamma \setminus
\Gamma'|$, and consider $$\mathcal U = \{ W \subseteq \Gamma : |W|
\geq |\Gamma| - y(n)\}.$$ Then, we have $|\mathcal U| \leq {n \choose
 y(n)}$, $\Gamma' \in \mathcal U$, and
 the collection $\mathcal U$ is completely determined
 by the rows of the adjacency matrix of $G$ corresponding to the vertices in $U$
 (as it is determined by $\Gamma = \bigcup_{u \in U} \Gamma(u)$).
 We may enumerate $\mathcal U = \{U_1, U_2, \ldots \}$ in such a way
 that each $U_i$ is determined after conditioning on $\Gamma$ (simply by
 canonically ordering all subsets of the vertices, and taking the induced
 ordering on $\mathcal U$).

 We can thus apply Theorem \ref{thm:l2mass2} to each
 of the ${n \choose y(n)}$ sets $U_i \in \mathcal U$ (one of
which will always be the set $\Gamma'$) and take a union bound
to obtain, for some $\beta = \beta(p) > 0$,
\begin{eqnarray*}
 \Pr[\exists U \textrm{ s.t. } \|f|_{\Gamma'}\|_2 =o(1) \cdot
 \|f|_{V\setminus U}\|_2 \textrm{ and } |\Gamma'| \geq (\tfrac12 +
 \varepsilon'_p) n] \leq {n \choose k_p + 1} {n \choose y(n)}
 \exp(-\beta n),
\end{eqnarray*}
and the latter quantity is $o(1)$ since $y(n) = o(n)$, but this
contradicts \eqref{eq:normbound37}.
\end{proof}

\begin{remark}
Observe that even under the preceding assumptions, we are not able
to rule out the case of one exceptional vertex $v$ with, say
$\Gamma(v) = \mathcal N_f$ and $\mathcal P_f = V \setminus
(\mathcal
 N_f \cup \{v\})$.
\end{remark}

\bigskip

%\newpage
\begin{center}
{\bf APPENDIX}
\end{center}

\appendix

\section{A few examples}
We used a simple MATLAB program to compute an upper bound
on the constant behind the $O_p(1)$ term in
\lemref{lem:fewsmall} for various $p$'s.
%The program finds
%$\varepsilon_1, \ldots, \theta, \xi_1, \xi_2$ for
%\eqref{eq:a1}, \eqref{eq:a2}, \eqref{eq:alpha} and \eqref{eq:beta}
%such that $\alpha$ and $\beta$ are positive and inequality
%\eqref{eq:H} holds.
It calculates the largest $k$ such that
inequality \eqref{eq:krdependence} holds, using $r = \tfrac{\alpha
  \sqrt{\tfrac12 + \varepsilon}}{2D}$,
  and the explicit bounds computed in Theorem
  \ref{thm:nonsymtail} and \eqref{eq:alpha} and \eqref{eq:beta}.
  %We found that $p$'s smaller
%than $0.18$ or larger than $0.8$ we cannot choose values for
%$\gamma$ and $\theta$ such that $\alpha>0$. The values
Our bounds for various $p$'s
%between $0.18$ and $0.78$
are given in the following table.

\begin{center}
 \begin{tabular}{|r|r|}
   \hline
   $p$ & $k$ \\
   \hline
   0.78 & 29 \\
   0.74 & 30 \\
   0.7 & 32 \\
   0.66 & 34 \\
   0.62 & 37 \\
   0.58 & 39 \\
   0.54 & 43 \\
   0.5 & 46 \\
   0.46 & 54 \\
   0.42 & 63 \\
   0.38 & 75 \\
   0.34 & 90 \\
   0.3  & 109 \\
   0.26 & 137 \\
   0.22 & 181 \\
   0.18 & 277 \\
    \hline
 \end{tabular}
\end{center}

\bibliography{mybib}
\bibliographystyle{plain}

\end{document}